\documentclass[a4paper]{amsart}
\usepackage{geometry}
\usepackage[foot]{amsaddr}

\usepackage[utf8]{inputenc} 
\usepackage[T1]{fontenc}    
\usepackage{url}            
\usepackage{booktabs}       
\usepackage{amsfonts}       
\usepackage{nicefrac}       
\usepackage{microtype}      
\usepackage{xcolor}         
\usepackage{multirow} 
\usepackage{cases}
\usepackage{amsmath,amssymb,amsthm,commath,esint,tikz-cd,tikz, mathtools, physics}
\usetikzlibrary{shapes,arrows,positioning}
\usepackage[mathscr]{euscript}
\usepackage{graphicx}
\usepackage{float}
\usepackage{fancyhdr}
\usepackage{bm}
\usepackage{bbm}
\usepackage[ruled,vlined]{algorithm2e}

\numberwithin{equation}{section}

\usetikzlibrary{external}
\usepackage{amsthm,amssymb}
\usepackage{mathtools,thmtools}
\usepackage{bm,esint}
\usepackage{color}
\usepackage{float,graphicx,subcaption}
\usepackage{cancel}
\usepackage[shortlabels]{enumitem}
\usepackage{mathrsfs}  
\usepackage{bbm}
\usepackage{euscript}
\usepackage{hyperref}
\usepackage{cleveref}
\usepackage{upgreek}

\usepackage{todonotes}

\renewcommand{\tilde}{\widetilde}
\renewcommand{\hat}{\widehat}

\renewcommand{\bar}{\overline}

\newcommand{\R}{\mathbb{R}}
\newcommand{\mC}{\mathbb{C}}
\newcommand{\N}{\mathbb{N}}

\newcommand{\cB}{\mathcal{B}}

\newcommand{\cG}{\mathcal{G}}

\newcommand{\cI}{\mathcal{I}}

\newcommand{\cL}{\mathcal{L}}

\newcommand{\bmalpha}{{\bm{\alpha}}}
\newcommand{\bmbeta}{{\bm{\beta}}}

\renewcommand{\hat}{\widehat}

\renewcommand{\epsilon}{\varepsilon}

\newcommand{\E}[1]{{\mathbb{E}\left[ #1 \right]}} 

\usepackage[OT2,T1]{fontenc}
\DeclareSymbolFont{cyrletters}{OT2}{wncyr}{m}{n}
\DeclareMathSymbol{\Sha}{\mathalpha}{cyrletters}{"58}

\newtheorem{theorem}{Theorem}[section]
\newtheorem{assumption}[theorem]{Assumption}
\newtheorem{example}[theorem]{Example}
\newtheorem{remark}[theorem]{Remark}

\newtheorem{lemma}[theorem]{Lemma}
\newtheorem{proposition}[theorem]{Proposition}

\title[Randomized neural networks for high-dimensional PDEs]{Approximation Theory and Applications of Randomized Neural Networks for Solving High-Dimensional PDEs}

\author{T.~De Ryck}

\author{S.~Mishra}

\author{Y.~Shang}

\author{F.~Wang}

 \address[T. De Ryck]{Seminar for Applied Mathematics, D-MATH, ETH Z\"urich, Rämistrasse 101, 8092 Zürich, Switzerland.  E-mail: {\tt  tim.deryck@sam.math.ethz.ch} }
 \address[S. Mishra]{Seminar for Applied Mathematics, D-MATH, and ETH AI Center, ETH Z\"urich, Rämistrasse 101, 8092 Zürich, Switzerland. E-mail: {\tt smishra@sam.math.ethz.ch} }
 \address[Y. Shang]{School of Mathematics and Statistics, Xi'an Jiaotong University, Xi'an, Shaanxi 710049, P.R. China. E-mail: {\tt fsy2503@stu.xjtu.edu.cn}}
 \address[F. Wang]{School of Mathematics and Statistics, Xi'an Jiaotong University, Xi'an, Shaanxi 710049, P.R. China. The work of this author was partially supported by the National Natural Science Foundation of China (Grant No. 92470115). E-mail: {\tt feiwang.xjtu@xjtu.edu.cn}}

\begin{document}

\maketitle

\begin{abstract}
We present approximation results and numerical experiments for the use of randomized neural networks within physics-informed extreme learning machines to efficiently solve high-dimensional PDEs, demonstrating both high accuracy and low computational cost. Specifically, we prove that RaNNs can approximate certain classes of functions, including Sobolev functions, in the  $H^2$-norm  at dimension-independent convergence rates, thereby alleviating the curse of dimensionality. Numerical experiments are provided for the high-dimensional heat equation, the Black-Scholes model, and the Heston model, demonstrating the accuracy and efficiency of randomized neural networks.
\end{abstract}

\section{Introduction}

Partial differential equations (PDEs) are vital tools to model phenomena in the sciences and engineering. As analytical solutions are rarely available, one must generally resort to the numerical approximation of the solutions of these PDEs. This has proven to be a very challenging task for high-dimensional PDEs as the computational cost of classical grid-based numerical method might become prohibitively expensive due to the curse of dimensionality. In recent years, multiple frameworks have risen in popularity that have the potential to alleviate this high cost. 

Neural networks are increasingly being used as ansatz spaces for approximating solutions to PDEs \cite{HEJ1, SZ1, Kuty, LMR1,LMPR1} as they are cheap to evaluate and tend to perform well in high-dimensional settings. Notably, they can provably overcome the curse of dimensionality in the approximation of high-dimensional PDEs such as the heat equation and Black-Scholes model \cite{hutzenthaler2020proof, grohs2018proof, jentzen2018proof}. However, training neural networks generally requires a large amount of (simulated) data, the generation of which is usually a major bottleneck in scientific computing. 

Physics-informed learning and most notably \emph{physics-informed neural networks (PINNs)} \cite{Lag1, Lag2, KAR1, KAR2} address this issue by using an unsupervised loss function based on the PDE residual, thereby eliminating the need for any training data. Moreover, PINNs can provably overcome the curse of dimensionality for e.g. linear Kolmogorov equations (such as the heat equation and Black-Scholes model) \cite{deryck2021pinn} and nonlinear parabolic PDEs \cite{deryck2022generic} and have been observed to perform well in other high-dimensional settings as well e.g. \cite{MM1,MM3,RT}. Despite their theoretical appeal, a more widespread usage of PINNs is limited due to the extremely difficult training process that is encountered for many PDEs \cite{krishnapriyan_characterizing, wang2021understanding,wang2022and}, which is related to the spectral properties of the differential operator \cite{deryck2023operator}. 

Due to these major challenges in the optimization of (physics-informed) neural networks in challenging settings, there has been a renewed interest in research focused on approximating PDE solutions using \emph{randomized neural networks (RaNNs) \cite{gallicchio2020deep}}. This concept refers to a class of neural networks where the connections to the hidden layer are fixed after random initialization, and only the output layer weights are adjusted. Neural networks with randomly sampled weights emerge in the work of Barron \cite{barron1993universal} and explored in random vector functional-links (RVFLs) \cite{pao1992functional}  and \emph{random feature models (RFMs) \cite{rahimi2007random}}, but have been popularized together with an efficient way to analytically determine the outer weights by Huang et al. \cite{huang2004extreme, huang2006extreme} under the name \emph{extreme learning machines (ELMs)}. Notable works on RFMs and ELMs include the analysis of Rahimi and Brecht e.g. \cite{rahimi2008uniform, rahimi2007random, rahimi2008weighted}, the adaptation of RFMs to data-driven surrogates for operators mapping between Banach spaces \cite{nelsen2021random}, the analysis of the approximation properties of random ReLU features e.g. \cite{sun2018approximation, gonon2020approximation, gonon2021random} and many more. 

Inspired by the success and popularity of both PINNs and ELMs, Dwidevi and Srinisavan \cite{dwivedi2020physics} have proposed to combine the unsupervised learning character of PINNs with the simplicity and low computational cost of ELMs by introducing \emph{physics-informed extreme learning machines (PIELM)}.  A method that combines the extreme theory of functional connections with single-layer neural networks was proposed in \cite{schiassi2021extreme}. Dong and Li \cite{dong2021local}  combines the ideas of local ELM and domain decomposition to improve accuracy and efficiency in locELM (local extreme learning machines). For problems with more complex geometries, an approach called the random feature method was proposed in \cite{chen2022bridging}. Additionally, RaNNs have been combined with various methods, including the Petrov-Galerkin formulation \cite{shang2023Randomized,shang2024randomized}, the discontinuous Galerkin method \cite{sun2024local,sun2024local_W}, and the hybrid discontinuous Petrov-Galerkin method \cite{Dang2024local}. These works all emphasize the potential of using randomized neural networks by demonstrating their empirical efficiency for low-dimensional PDEs. To the best of our knowledge, there are no available results on the approximation error of using randomized neural networks in high-dimensional PDEs.

This works attempts to address this paucity in the literature. As a first contribution, we establish upper bounds on the approximation error of randomized neural networks with tanh activation function in \emph{higher-order Sobolev norms} ($H^1$ and $H^2$). For linear first-order  and second-order PDEs $\cL[u]=0$ this will imply that the physics-informed loss will be small as well, given that $\norm{\cL[u_\theta]}_{L^2}\lesssim \norm{u_\theta-u}_{H^k}$ for any model $u_\theta$. Our proofs are constructive and draw inspiration from a representation formula proven in \cite{gonon2020approximation} and refined in \cite{gonon2021random}. This representation formula allowed Gonon \cite{gonon2021random} to prove that randomized neural networks can overcome the curse of dimensionality in the approximation of the solution to the Black-Scholes model and more general Lévy models in supremum norm, meaning that the networks size scales at most polynomially in the PDE dimension $d$. For an error tolerance of $\epsilon>0$ this means that the network size is at most $O(d^\alpha \epsilon^{-\beta})$ for $\alpha,\beta>0$ independent of $d$. 

The main theoretical contributions of this work are as follows: 
\begin{itemize}
    \item In Theorem \ref{thm:approx-general} in Section \ref{sec:approx1} we prove for functions of the form
    \begin{equation}
        u(x) = \int_{\R^d} e^{i x \cdot \xi} G(\xi) d\xi
    \end{equation}
    that they can be approximated by shallow randomized neural networks of width $N\in\N$ at a dimension-independent rate of $N^{-1/4}$ in $H^1$-norm and $N^{-1/10+\epsilon}$ in $H^2$-norm, $\epsilon>0$, under some integrability conditions for the probability distributions of the random weights and biases as well as $G$. 
    \item In Theorem \ref{thm:approx-sobolev} in Section \ref{sec:approx2} this result is then used to prove an approximation result for shallow randomized neural networks with uniformly distributed random weights and biases specifically for Sobolev functions. We show that the curse of dimensionality can be fully overcome if sufficient regularity is assumed, with convergence rates that tend to those of the more general Theorem \ref{thm:approx-general} for highly regular functions. 
\end{itemize}
Finally, we expand the literature as well by demonstrating that randomized neural networks can overcome the curse of dimensionality in practice for high-dimensional in PDEs. More precisely: 
\begin{itemize}
    \item In Section \ref{sec:4} we present numerical experiments demonstrating the accuracy and efficiency of randomized neural networks for the high-dimensional heat equation (Section \ref{sec:heat}), Black-Scholes model (Section \ref{sec:bs}) and Heston model (Section \ref{sec:heston}). We let the dimension range from 1 to 100 and show that even in very-high dimensional settings physics-informed extreme learning machines work well and fast. For the heat equation we record $L^2$-errors between $10^{-4} \%$ and $2.4\%$ for computing times between 1 and 29 seconds. For the more difficult Black-Scholes and Heston models we record relative errors of a one to three percent and computation times between 10 seconds and 6 minutes. 
\end{itemize}

\section{Preliminaries}\label{sec:2}

We first give a short introduction to randomized neural networks as a specific type of random feature models (Section \ref{sec:random-nn}), extreme learning machines (Section \ref{sec:elm}), physics-informed neural networks (Section \ref{sec:piml}) and physics-informed extreme learning machines (Section \ref{sec:pielm}). 

\subsection{Randomized neural networks}\label{sec:random-nn}

Given a probability distribution $\nu$ on the space of square-integral functions $L^2(D; \R)$ on domain $D\subset \R^d$, a general \emph{random feature model} is defined as the weighted sum of a number of i.i.d. generated functions $\varphi_1, \ldots \varphi_N \sim \nu$, that is
\begin{equation}\label{eq:rfm}
    U_W: \R^d \to \R: x\mapsto \sum_{i=1}^N W_i \varphi_i(x),
\end{equation}
where the weights $W_1, \ldots W_N\in \R$ can be chosen freely, and will be optimized such that $U_W \approx u$. Often, one chooses parametrized functions $\varphi_i(\cdot) := \varphi(\cdot; \theta_i)$ as random features, such that only (finite-dimensional) parameters $\theta_1, \ldots \theta_N \sim \nu^*$ need to be drawn. 

In this article, we study \emph{randomized neural networks}, i.e. the case where random features are parametrized as feedforward neural network with one hidden layer. In this setting, the definition of a randomized neural network \eqref{eq:rfm} can be specialized to
\begin{equation}\label{eq:rnn}
    U_W^{A,B}: \R^d \to \R: x\mapsto \sum_{i=1}^N W_i \sigma(A_i\cdot x+B_i),
\end{equation}
where we let $\sigma:\R\to \R$ is a nonlinear activation function, we let $A_1, \ldots A_N\in \R^d$ be random weight vectors and we let $B_1, \ldots B_N\in \R$ be random bias scalars; all of which are independent. The goal is to choose the output weights $W=(W_1, \ldots W_N)$ as a function of the random hidden weights $A=(A_1, \ldots A_N)$ and $B=(B_1, \ldots, B_N)$ such that $U_W^{A,B}$ is a good approximation of $u$. We note that one can also add an additional bias $W_0$ to definitions \eqref{eq:rfm} and \eqref{eq:rnn}. 

\subsection{Extreme learning machines}\label{sec:elm}

Huang et al. \cite{huang2004extreme} proposed to analytically determine the weights of the above random single-hidden layer feedforward neural networks (SLFNs) and coined the name ELM for SLFNs. Given access to a training set $\{(x_1, u(x_1)), \ldots (x_n, u(x_n))\}$ the goal is to find the weight vector $W\in \R^N$ such that ideally
\begin{equation}\label{eq:elm1}
    \sum_{i=1}^N W_i \sigma(A_i\cdot x_k +B_i) = u(x_k), \qquad \forall k\in \{1, \ldots, n\}. 
\end{equation}
If we define matrix $H\in \R^{n\times N}$ with $H_{ki} = \sigma(A_i\cdot x_k +B_i)$ and a vector $T\in \R^n$ with $T_k = u(x_k)$ then the above $n$ equations \eqref{eq:elm1} can be written compactly as
\begin{equation}\label{eq:elm2}
    HW = T. 
\end{equation}
Rather than solving this with gradient-based learning algorithms, the ELM method proposes to choose $\hat{W} = H^\dagger T$, where $H^\dagger$ is the Moore-Penrose generalized inverse of $H$. Hence, $\hat{W}$ is the smallest norm least-squares solution of the linear system \eqref{eq:elm2}. As such, the ELM framework provides a simple and extremely fast alternative to training neural networks with gradient-based optimizers. 

\subsection{Physics-informed machine learning}\label{sec:piml}

Physics-informed machine learning is a popular framework in order to choose a model $u_\theta: D\to \R$ in a parametrized ansatz space $\{u_\theta: \theta\in\Theta\}$ with the goal to approximate the solution $u:D\to\R$ of a partial differential equation (PDE), as given by e.g.
\begin{equation}\label{eq:pde}
    \cL[u] = 0 \qquad \text{and} \qquad \cB[u] = 0, 
\end{equation}
where a \emph{differential operator} $\cL$ prescribes the PDE and a \emph{boundary operator} $\cB$ prescribes the boundary conditions and initial condition. The crux of physics-informed learning is that the selection of the optimal model in the ansatz space is done by minimizing the \emph{physics-informed loss function} $L(\theta)$, or a discretization thereof, which is defined as
\begin{equation}\label{eq:piml}
    L(\theta) = \int_D (\cL[u_\theta](x))^2 dx + \lambda \int_{\partial D} (\cB[u_\theta](y))^2 dy,
\end{equation}
where $\lambda>0$ is a hyperparameter. Alternatively, one can adapt the ansatz space in such a way that all models $u_\theta$ satisfy the boundary conditions, i.e. $\cB[u_\theta]=0$ for all $\theta\in\Theta$. In this case we speak of \emph{hard} boundary conditions and the second term on the RHS of \eqref{eq:piml} can be ignored. A physics-informed loss is most commonly used in the combination with neural networks, leading to the popular framework of \emph{physics-informed neural networks (PINNs)}. The main advantage of using a physics-informed learning framework is that training can be done fully unsupervised, without the need for measurements are potentially expensive simulated data. It is also useful in the setting of complex geometries as the method does not require any grid. 

\subsection{Physics-informed extreme learning machines}\label{sec:pielm}

Given the simplicity and speed of extreme learning machines and the unsupervised character of physics-informed machine learning, it comes to no surprise that the combination of those two frameworks has caught the attention of various researchers. It was first coined as \emph{physics-informed extreme learning machine (PIELM)} by \cite{dwivedi2020physics}. We describe the PIELM framework for linear PDEs with linear boundary conditions. Given two sets of multi-indices $\Lambda, B \subset  \N^d_0$, a linear PDE can be written as
\begin{equation}
    \cL[u](x) = \sum_{\bmalpha\in \Lambda} a_\bmalpha(x) D^\bmalpha u(x) - f(x),\qquad  \cB[u](x) = \sum_{\bmbeta\in B} b_\bmbeta(x) D^\bmbeta u(x) - g(x)
\end{equation}
where $D^\bmalpha = \partial^{\bmalpha_1}_{x_1}\cdots \partial^{\bmalpha_d}_{x_d}$ and $a_\bmalpha:D\to\R$, $b_\bmbeta:D\to\R$ are some functions. As a result, if we discretize the physics-informed loss $L(\theta)$ \eqref{eq:piml}  with $n_x$ points in $D$ and $n_y$ points in $\partial D$, then its minimization is equivalent to solving the following linear system in a least-square sense,
\begin{equation}\label{eq:pielm}
\begin{cases}
    \sum_{i=1}^N \sum_{\bmalpha\in \Lambda} a_\bmalpha(x_k) W_i A_i^\bmalpha \sigma^{\norm{\bmalpha}_1}(A_i\cdot x_k +B_i) &= f(x_k), \qquad \forall k\in \{1, \ldots, n_x\}, \\
    \sum_{i=1}^N \sum_{\bmbeta\in B} b_\bmbeta(y_k) W_i A_i^\bmbeta \sigma^{\norm{\bmbeta}_1}(A_i\cdot y_k +B_i) &= g(y_k), \qquad \forall k\in \{1, \ldots, n_y\}, 
\end{cases}
\end{equation}
where $\{x_1, \ldots x_{n_x}\}\subset D$ and $\{y_1, \ldots, y_{n_y}\}\subset \partial D$ are the chosen collocation points. This can be written in the form $HW=T$, such as in \eqref{eq:elm2}, if one defines $H\in \R^{(n_x+n_y)\times N}$ and $T\in \R^{n_x+n_y}$ as
\begin{align}
    &H_{k,i} = \sum_{\bmalpha\in \Lambda} a_\bmalpha(x_k) A_i^\bmalpha \sigma^{\norm{\bmalpha}_1}(A_i\cdot x_k +B_i), \qquad T_{k} = f(x_k)\qquad \forall k\in \{1, \ldots, n_x\}, \\
    &H_{n_x+k,i} = \sum_{\bmbeta\in B} b_\bmbeta(y_k)A_i^\bmbeta \sigma^{\norm{\bmbeta}_1}(A_i\cdot y_k +B_i) , \qquad T_{n_x+k} = g(y_k)\qquad \forall k\in \{1, \ldots, n_y\}. 
\end{align}
The proposed physics-informed ELM solution would then be again $\hat{W} = H^\dagger T$. In what follows, we show approximation results for randomized neural networks in Sobolev norms, as well as experiments demonstrating the speed and accuracy of the PIELM framework. 

\section{Approximation results for randomized NNs in Sobolev norms}\label{sec:3}

In this section, we show that randomized neural networks can approximate Sobolev functions well in higher-order Sobolev norms. In particular, the curse of dimensionality can be alleviated in the sense that dimension-independent convergence rates in terms of the size of the network can be obtained. 

\subsection{Function approximation with generalized randomized neural networks}\label{sec:approx1}

To prepare for the main results of the next subsection, we first consider a \emph{generalized} definition of randomized neural networks, namely of the form
\begin{equation}\label{def:rfm-1}
    U^{A,B,Y}_W(x) = W_0 + \sum_{i=1}^N W_i \sigma(A_i\cdot x+B_i - Y_i),
\end{equation}
where $\sigma$ is tanh activation function. In addition, we make the following assumptions on the distribution of the hidden weights of the randomized neural network:  
\begin{assumption}\label{ass:dist-1}
Let $M\geq 1$. The random variables $A_i, B_i$ and $Y_i$ ($i=1\ldots,N$) are iid and satisfy the following requirements:
\begin{itemize}
	\item the distribution of $A_1$ has a strictly positive Lebesgue-density $\pi_A$ on $\R^d$ and
	\item the distribution of $B_1$ has a strictly positive Lebesgue-density $\pi_B$ on $\R$ and
    \item the distribution of $Y_1$ is uniform on $[0,2M\norm{A_1}_1+1]$.
\end{itemize}
\end{assumption}

\begin{assumption}\label{ass:piA-radial}
The density $\pi_A$ is radially symmetric, meaning that there exists $\pi_A^*:\R\to\R$ such that $\pi_A^*(\norm{\xi}_2) = \pi_A(\xi)$ for all $\xi\in\R^d$. 
\end{assumption}

\begin{assumption}\label{ass:Fbar}
It holds that $\bar{F}(r) := 2\int_{-r}^0 \frac{1}{\pi_B(s)} ds \in (-\infty,\infty)$ for all $r \in \R$. 
\end{assumption}
Next, we introduce the function
 \begin{equation}
     H_\epsilon(x) = \frac{1}{2}(1+\sigma(x/\epsilon)) = \frac{1}{2}(1+\tanh(x/\epsilon))
 \end{equation}
for any $\epsilon>0$ as a smooth approximation of the Heaviside function $H$. We also define $\mu$ as the measure of the uniform distribution on $B_M^d := \{x\in\R^d: \norm{x}_2 \leq M\}$.

 With this notation in place, we will adapt an approximation result from \cite{gonon2021random} for randomized neural networks with ReLU activation to a representation formula in the form of an infinite-width tanh randomized  neural network. 

\begin{proposition}\label{prop:general}
Let $u \colon \R^d \to \R$, let $M \geq 1$ and assume there exists $G \colon \R^d \to \mC$ such that 
	\begin{equation}
	\label{eq:Hrepresentation2}
	u(x) = \int_{\R^d} e^{i x \cdot \xi} G(\xi) d\xi
	\end{equation}
	for all $x \in [-M,M]^d$. For any $\epsilon>0$ we define $u_\epsilon:\R\to\R$ as
\begin{equation}
      u_\epsilon(x) = \int_{\R^d} \int_{- \infty}^\infty \int_0^{x \cdot \xi+u} H_\epsilon(y)  \alpha(\xi,u) d y d u d \xi,
 \end{equation}
where $\alpha$ is defined in terms of $\tilde{g}(\xi) := 2\mathrm{Re}[G](\xi)-\mathrm{Im}[G](\xi)$ and $G$ as 
 \begin{equation}
     \alpha(\xi,u) = -\mathbbm{1}_{(-M \|\xi\|_1,0]}(u)\mathrm{Re}[e^{-i u} G(\xi) + e^{i u} G(-\xi)] + \mathbbm{1}_{[0,1]}(u)\tilde{g}(\xi)-\mathbbm{1}_{[-1,0]}(u)\tilde{g}(-\xi).
 \end{equation}
If Assumption \ref{ass:piA-radial} is satisfied, then for any $\epsilon>0$ and any $\gamma\in(0,1/2)$ there exists a constant $C_\gamma$ such that  
\begin{equation}
    \norm{u-u_\epsilon}_{W^{1,\infty}(B_M^d)} \leq 10M\sqrt{d}\cI_2^* \epsilon \quad \text{and}\quad \norm{u-u_\epsilon}_{H^2_\mu(B_M^d)} \leq C_\gamma d(\pi_A(0)+d)M\cI_4^*\epsilon^\gamma,
\end{equation}
where 
	\begin{equation}
	\label{eq:IfiniteLinfty-star}
	\cI_\ell^* := \max_{k\in \{0, \ldots, \ell \}}\left(\int_{\R^d} \|\xi\|^k \frac{(\abs{G(\xi)}+\abs{G(-\xi)})^2}{ \pi_A(\xi)} d \xi\right)^{\frac{1}{2}}. 
	\end{equation}   
\end{proposition}

\begin{proof}
\textbf{Step 1: construction of approximation.} 
The starting point of the proof is a representation formula that can be found in equation (10) in \cite{gonon2021random}. It states that for all $x \in [-M,M]^d$ it holds that
	\begin{equation}
	\label{eq:u-representation1}
	\begin{aligned}
	u(x) 
	& = \int_{\R^d} \int_{- \infty}^\infty  (x \cdot \xi+u)_+  \alpha(\xi,u) d u d \xi
	\end{aligned}
	\end{equation}
where $\alpha$ is defined in terms of $\tilde{g}(\xi) := 2\mathrm{Re}[G](\xi)-\mathrm{Im}[G](\xi)$ and $G$ as 
 \begin{equation}
     \alpha(\xi,u) = -\mathbbm{1}_{(-M \|\xi\|_1,0]}(u)\mathrm{Re}[e^{-i u} G(\xi) + e^{i u} G(-\xi)] + \mathbbm{1}_{[0,1]}(u)\tilde{g}(\xi)-\mathbbm{1}_{[-1,0]}(u)\tilde{g}(-\xi).
 \end{equation}

In \cite{gonon2021random} the representation formula \eqref{eq:u-representation1} is used to prove the existence of a ReLU neural network that approximates $u$ in $L^\infty$-sense. We will use the function $H_\epsilon$ as a smooth approximation of the Heaviside function $H$. 
 Next, we note from \eqref{eq:u-representation1} that for all $x\in[-M,M]^d$ it holds on the support of $\alpha$ that
\begin{equation}\label{eq:1-10}
    \abs{x\cdot \xi+u} \leq 2M \|\xi\|_1+1 =: B(\xi).
\end{equation}
These observations let us rewrite \eqref{eq:u-representation1} as
\begin{equation}
    u(x) = \int_{\R^d} \int_{- \infty}^\infty \int_0^{x \cdot \xi+u} H(y)  \alpha(\xi,u) d y d u d \xi = \int_{\R^d} \int_{- \infty}^\infty \int_0^{B(\xi)} H(x \cdot \xi+u - y)  \alpha(\xi,u) d y d u d \xi.
\end{equation}
This inspires us to define
\begin{equation}\label{eq:u-epsilon}
    u_\epsilon(x) = \int_{\R^d} \int_{- \infty}^\infty \int_0^{x \cdot \xi+u} H_\epsilon(y)  \alpha(\xi,u) d y d u d \xi = \int_{\R^d} \int_{- \infty}^\infty \int_0^{B(\xi)} H_\epsilon(x \cdot \xi+u - y)  \alpha(\xi,u) d y d u d \xi.
\end{equation}

\textbf{Step 2: accuracy of approximation.} 

\textit{Step 2a: preparation.} In the next steps, we will encounter many quantities that are defined in terms of $\alpha$ and $G$. Let us first introduce the notation
\begin{equation}
    \cG(\xi) = \abs{G(\xi)}+\abs{G(-\xi)} \quad \text{ and } \beta(\xi,u) = \mathrm{Re}[e^{-i u} G(\xi) + e^{i u} G(-\xi)]. 
\end{equation}
One can then calculate that 
\begin{equation}
    \partial_u \beta(\xi, u) = \mathrm{Re}[e^{-i (u-\pi/2)} G(\xi) + e^{i (u-\pi/2)} G(-\xi)]. 
\end{equation}
Recalling the definitions of $\alpha$ and $\beta$ we find that,
\begin{equation}\label{eq:abs-alpha-bound}
    \abs{\alpha(\xi,u)}\leq (\mathbbm{1}_{(-M \|\xi\|_1,0]}(u) + 4\mathbbm{1}_{[-1,1]}(u))\cG(\xi)
\end{equation}
and hence,
\begin{equation}
    \int_{\R^d}\int_{\R} \abs{\alpha(\xi,u)}dud\xi \leq \int_{\R^d}(M\norm{\xi}_1+8)\cG(\xi) d\xi \leq 9M\sqrt{d}\int_{\R^d}\max\{1,\norm{\xi}_2\}\cG(\xi) d\xi,
\end{equation}
and consequently, 
\begin{equation}\label{eq:1-17}
    \max\{\norm{\alpha(\xi, \cdot)}_{L^\infty(\R)},\norm{\partial_u\beta(\xi, \cdot)}_{L^\infty(\R)}\} \leq 5\cG(\xi).
\end{equation}
We also have that from Hölder's inequality,
\begin{equation}
    \int_{\R^d} \norm{\xi}_2^{k} \cG(\xi) d\xi  = \int_{\R^d} \frac{\norm{\xi}_2^{k} \cG(\xi)}{\sqrt{\pi_A(\xi)}}\cdot \sqrt{\pi_A(\xi)} d\xi \leq \left(\int_{\R^d}\frac{\norm{\xi}_2^{2k} \cG(\xi)^2}{{\pi_A(\xi)}}d\xi\right)^{\frac{1}{2}} \leq \cI_{2k}^*,
\end{equation}
where we define for any $\ell\in\N_0$ the constants
\begin{equation}
\begin{split}
	\label{eq:IfiniteLinfty-ell}
 \cI_\ell^* &:= \max_{k\in \{0, \ldots, \ell \}}\left(\int_{\R^d} \|\xi\|^k \frac{\cG(\xi)^2}{ \pi_A(\xi)} d \xi\right)^{\frac{1}{2}},\\
	\cI_\ell &:= \max_{k\in \{0, \ldots, \ell \}}\left(\int_{\R^d} \|\xi\|^k [ 1+\bar{F}(M\|\xi\|_1) +  \bar{F}(1)-\bar{F}(-1) ] \frac{\cG(\xi)^2}{ \pi_A(\xi)} d \xi\right)^{\frac{1}{2}}.
\end{split}
\end{equation}   

\textit{Step 2b: error introduced by $H\approx H_\epsilon$. }
We now investigate the accuracy of this approximation. First we observe that for $x>0$ it holds
\begin{equation}
    \abs{1-H_\epsilon(x)} = \frac{1}{2}(1-\sigma(x/\epsilon)) = \frac{1}{2} \frac{2\exp(-x/\epsilon)}{\exp(x/\epsilon)+\exp(-x/\epsilon)} \leq \exp(-x/\epsilon)
\end{equation}
and hence, by symmetry, also $\abs{H_\epsilon(x)}\leq \exp(x/\epsilon)$ for $x<0$. Using this we find that, 
\begin{equation}
    \begin{split}
        \abs{u(x)-u_\epsilon(x)} &\leq \int_{\R^d} \int_{- \infty}^\infty \int_0^{x \cdot \xi+u} \abs{H(y)-H_\epsilon(y)  }\abs{\alpha(\xi,u)} d y d u d \xi \\
        &\leq \int_{\R^d} \int_{- \infty}^\infty \int_0^{\infty} \exp(-x/\epsilon) \abs{\alpha(\xi,u)} d y d u d \xi \\
        &= \epsilon  \int_{\R^d} \int_{- \infty}^\infty  \abs{\alpha(\xi,u)} d u d \xi \\
        &\leq 9M\sqrt{d}{\cI_2}^*\epsilon . 
    \end{split}
\end{equation}
Next, we investigate the error in approximating the first derivative by choosing $\ell\in \{1, \ldots, d\}$ and calculating
\begin{equation}
    \begin{split}
        \abs{\partial_\ell(u(x)-u_\epsilon(x))} &\leq \int_{\R^d} \int_{- \infty}^\infty \abs{\xi_\ell} \abs{H(x \cdot \xi+u)-H_\epsilon(x \cdot \xi+u)  }\abs{\alpha(\xi,u)} d u d \xi \\
        &\leq \int_{\R^d} \abs{\xi_\ell} \norm{\alpha(\xi,\cdot)}_{L^\infty(\R)} \int_{- \infty}^\infty  \abs{H(u)-H_\epsilon(u)  } d u d \xi \\
&\leq 5 \int_{\R^d} \abs{\xi_\ell} \cG(\epsilon) \int_{- \infty}^\infty  \exp(-\abs{u}/\epsilon) d u d \xi \\
&= 10\epsilon \int_{\R^d} \norm{\xi}_2 \cG(\epsilon) d\xi \\
&\leq 10{\cI_2}^*\epsilon. 
    \end{split}
\end{equation}
These two calculations already tell us that
\begin{equation}
    \norm{u-u_\epsilon}_{W^{1,\infty}(B_M^d)} \leq 10M\sqrt{d}\cI_2^* \epsilon.
\end{equation}
Finally, we consider the second derivative for $j,\ell \in \{1, \ldots, d\}$. We first note that
\begin{equation}
\begin{split}
    \partial_j\partial_\ell u(x) &= \partial_j \int_{\R^d} \int_{-\infty}^\infty \xi_\ell H(x \cdot \xi+u)\alpha(\xi,u) du d\xi \\
&= \partial_j \int_{\R^d} \int_{-x\cdot\xi}^\infty \xi_\ell \alpha(\xi,u) du d\xi \\
    &= \int_{\R^d}  \xi_j \xi_\ell \alpha(\xi, -x\cdot \xi) d\xi. 
\end{split}
\end{equation}
Given that $\int_{-\infty}^\infty \abs{H_\epsilon'(x)} dx = 1$ 
we can write
\begin{equation}
    \begin{split}
\abs{\partial_j\partial_\ell(u(x)-u_\epsilon(x))} &\leq \int_{\R^d} \int_{- \infty}^\infty \abs{\xi_j \xi_\ell }\abs{H_\epsilon'(x \cdot \xi+u)  }\abs{\alpha(\xi,u)-\alpha(\xi,-x\cdot \xi)} d u d \xi. 
    \end{split}
\end{equation}
We will analyze this error by splitting the domain of integration into three parts. To create this subdivision we define the set
\begin{equation}
    A_{x,\delta} := \{\xi\in \R^d: \abs{x^*+x\cdot\xi}>\delta \text{ for }x^*±\in \{-1,0,1\}\}. 
\end{equation}
The intuition of this set is that all jump discontinuities in the map $\xi\mapsto \alpha(\xi, -x\cdot \xi)$ are contained in the complement of this set, as the map $u\mapsto \alpha(\xi,u)$ has jump continuities when $u\in\{-1,0,1\}$. Whenever $\xi\in A_{x,\delta}$, however, we can use Taylor's theorem to find a $\zeta(u,\xi,x)\in \R$ such that
\begin{equation}
    \abs{\alpha(\xi,u)-\alpha(\xi,-x\cdot \xi)} \leq \abs{\beta(\xi,u)-\beta(\xi,-x\cdot \xi)} \leq \abs{u+x\cdot \xi} \abs{\partial_u \beta (\xi, \zeta)}. 
\end{equation}
Without loss of generality we only consider $x^*=0$ in the definition of $A_{x,\delta}$, as the other cases are analogous. Under this assumption, and using \eqref{eq:1-17}, we see that
\begin{equation}
\begin{split}
    &\int_{\R^d\cap A_{x,\delta}} \int_{-x\cdot \xi -\delta}^{-x\cdot \xi +\delta}  \abs{\xi_j \xi_\ell } \abs{H_\epsilon'(x \cdot \xi+u)  }\abs{\alpha(\xi,u)-\alpha(\xi,-x\cdot \xi)} d u d \xi \\ 
    &\quad\leq \frac{\norm{\sigma'}_\infty}{\epsilon}\int_{\R^d}  \abs{\xi_j \xi_\ell } \int_{-x\cdot \xi -\delta}^{-x\cdot \xi +\delta}  \abs{u+x\cdot \xi} \abs{\partial_u \beta (\xi, \zeta)}  d u d \xi \\
    &\quad\leq \frac{\norm{\sigma'}_\infty}{\epsilon}\int_{\R^d}  \abs{\xi_j \xi_\ell } \norm{\partial_u\beta(\xi,\cdot)}_{L^\infty(\R)} \int_{-\delta}^\delta \abs{u} du d\xi \\
&\quad\leq \frac{5\delta^2\norm{\sigma'}_\infty}{\epsilon}\int_{\R^d} \norm{\xi}^2_2 \cG(\xi) d\xi \\
&\quad\leq \frac{5\delta^2\norm{\sigma'}_\infty}{\epsilon} {\cI_4}^*.
\end{split}
\end{equation}
Next we analyze the error on the complement of $A_{x,\delta}$ in the vicinity of $u = -x\cdot \xi$, still assuming wlog that $x^*=0$ is the only relevant case,
\begin{equation}
\begin{split}
    &\int_{\R^d\cap A^c_{x,\delta}} \int_{-x\cdot \xi -\delta}^{-x\cdot \xi +\delta} \abs{\xi_j \xi_\ell } \abs{H_\epsilon'(x \cdot \xi+u)  }\abs{\alpha(\xi,u)-\alpha(\xi,-x\cdot \xi)} d u d \xi \\ 
    &\quad\leq \frac{2\norm{\sigma'}_\infty}{\epsilon}\int_{\R^d\cap A^c_{x,\delta}} \abs{\xi_j \xi_\ell }  \norm{\alpha(\xi,\cdot)}_{L^\infty(\R)} \int_{-x\cdot \xi -\delta}^{-x\cdot \xi +\delta} d u d \xi \\
    &\quad= \frac{4\delta\norm{\sigma'}_\infty}{\epsilon}\cdot \int_{\R^d\cap A^c_{x,\delta}} \abs{\xi_j \xi_\ell }  \norm{\alpha(\xi,\cdot)}_{L^\infty(\R)}  d \xi =: (\star). 
\end{split}
\end{equation}
We now aim to provide an upper bound for the integral above. 
Using Hölder's inequality we find
\begin{equation}\label{eq:H2-2z}
    {\int_{\R^d\cap A^c_{x,\delta}} \abs{\xi_j \xi_\ell  }\norm{\alpha(\xi,\cdot)}_{L^\infty(\R)}  d \xi} \leq \left(\int_{\R^d\cap A^c_{x,\delta}} \pi_A^*(\norm{\xi}) d\xi\right)^{\frac{1}{2}} \left(\int_{\R^d} \frac{\norm{\xi}^4  \norm{\alpha(\xi,\cdot)}_{L^\infty(\R)}^2}{\pi_A^*(\norm{\xi})}  d \xi\right)^{\frac{1}{2}}. 
\end{equation}
From this, \eqref{eq:1-17} and \eqref{eq:IfiniteLinfty-ell} we find the elementary upper bound
\begin{equation}\label{eq:H2-2a}
    {\int_{\R^d\cap A^c_{x,\delta}} \abs{\xi_j \xi_\ell } \norm{\alpha(\xi,\cdot)}_{L^\infty(\R)}  d \xi} \leq \left(\int_{\R^d} \frac{\norm{\xi}^4  \norm{\alpha(\xi,\cdot)}_{L^\infty(\R)}^2}{\pi_A^*(\norm{\xi})}  d \xi\right)^{\frac{1}{2}} \leq 5\cI_4^*. 
\end{equation}
Next, we prove a second upper bound again using \eqref{eq:H2-2z}. 
We remark that (for fixed $x\neq 0$) the shortest distance between the two hyperplanes $x\cdot \xi = \pm \delta$ is $2\delta / \norm{x}_2$. For simplicity, and without loss of generality, we assume that the direction of $(\xi_1, 0, \ldots, 0)$ is perpendicular to the hyperplanes. In this case we find for $d\geq2$ that
\begin{equation}\label{eq:H2-2b}
\begin{split}
\int_{\R^d\cap A^c_{x,\delta}} \pi_A^*(\norm{\xi}) d\xi &\leq \frac{2\delta}{\norm{x}_2} \max_{\xi_1}\int_{\R^{d-1}} \pi_A^*(\norm{\xi}) d\xi \\&\leq \frac{2\delta}{\norm{x}_2} \int_{\R^{d-1}} \pi_A^*(\norm{\xi_{2:d}}) d\xi_{2:d}\\
&\leq \frac{2\omega_{d-1}\delta}{\norm{x}_2} \int_{0}^\infty \pi_A^*(r) r^{d-2}dr \\
&\leq \frac{2\omega_{d-1}\delta}{\norm{x}_2} \left[\int_{0}^1 \pi_A^*(r) r^{d-2}dr + \int_{1}^\infty \pi_A^*(r) r^{d-2}dr\right]\\
&\leq \frac{2\omega_{d-1}\delta}{\norm{x}_2} (\pi_A^*(0)+\omega_d^{-1}),
\end{split}
\end{equation}
where in the last line we used that
\begin{equation}
    \int_{1}^\infty \pi_A^*(r) r^{d-2}dr \leq \int_{0}^\infty \pi_A^*(r) r^{d-1}dr = \omega_d^{-1} \int_{\R^d} \pi_A(\xi) d\xi = \omega_d^{-1}.
\end{equation}
Note that these inequalities also hold for more any orientation of the hyperplanes $x\cdot \xi = \pm \delta$. Also note that for $d=1$ the inequality simply becomes
\begin{equation}
    \int_{\R^d\cap A^c_{x,\delta}} \pi_A^*(\norm{\xi}) d\xi \leq \frac{2\delta\pi_A^*(0)}{\norm{x}_2} \leq \frac{2\omega_{d-1}\delta}{\norm{x}_2} (\pi_A^*(0)+\omega_d^{-1}),
\end{equation}
as $\omega_0=1$. We can now combine the two versions of the upper bound (i.e. \eqref{eq:H2-2a} and \eqref{eq:H2-2z}+\eqref{eq:H2-2b}) to calculate its $L^2_\mu(B_M^d)$-norm. 
As a preparatory calculation, and assuming that $\delta<\frac{M}{e}$, we find for $d=1$ that
\begin{equation}
   \frac{1}{2M}\int_{-M}^M \min\left\{1,\frac{1}{\norm{x}}\right\} dx \leq \frac{1}{M}\int_0^\delta  dr + \frac{1}{M} \int_\delta^M \frac{1}{r} dr \leq \frac{2}{M} \ln(\frac{M}{\delta}).
\end{equation}
and for $d\geq 2$ we find that
\begin{equation}
    \frac{1}{\omega_dM^d}\int_{B_M^d} \min\left\{1,\frac{1}{\norm{x}}\right\} dx \leq \frac{1}{M^d}\int_0^M \frac{r^{d-1}}{r} dr \leq \frac{1}{M(d-1)}\leq \frac{2}{M} \ln(\frac{M}{\delta}). 
\end{equation}
Hence, combining all ingredients and using $\omega_{d-1}(\pi_A^*(0)+\omega_d^{-1})\leq6\pi_A(0)+5d$, still assuming that $\delta<\frac{M}{e}$, we find that
\begin{equation}
    \begin{split}
        \left(\int_{B_M^d}(\star)^2 d\mu(x)\right)^{1/2} &\leq \frac{5\cI_4^* \cdot 4\norm{\sigma'}_\infty\delta}{\epsilon}  \left( \frac{2\omega_{d-1}\delta(\pi_A^*(0)+\omega_d^{-1})\cdot 2}{M} \ln(\frac{M}{\delta}) \right)^{\frac{1}{2}} \\
        &\leq \frac{80\cI_4^* \delta^{3/2}}{\epsilon} (6\pi_A(0)+5d)  \ln(\frac{1}{\delta}), 
    \end{split}
\end{equation}
where we used that if $M\geq 1$ and $\delta<\frac{M}{e}$ then
\begin{equation}
    \frac{2}{M} \ln(\frac{M}{\delta}) = \frac{2}{M} \ln(M) + \frac{2}{M} \ln(\frac{1}{\delta}) \leq 2+2 \ln(\frac{1}{\delta}) \leq 4 \ln(\frac{1}{\delta}).
\end{equation}
Finally, using amongst others \eqref{eq:1-17}, we analyze the error on the remaining subdomain: 
\begin{equation}
    \begin{split}
        &\int_{\R^d} \int_{\R\setminus[-x\cdot \xi \pm\delta]} \abs{\xi_j \xi_\ell }\abs{H_\epsilon'(x \cdot \xi+u)  }\abs{\alpha(\xi,u)-\alpha(\xi,-x\cdot \xi)} d u d \xi \\ 
        &\quad \leq 4 \int_{\R^d} \abs{\xi_j \xi_\ell}\norm{\alpha(\xi,\cdot)}_{L^\infty(\R)} \int_\delta^\infty  \abs{H_\epsilon'(u)  } du d\xi \\
        &\quad \leq 4 (1-H_\epsilon(\delta)) \int_{\R^d} \abs{\xi_j \xi_\ell}\norm{\alpha(\xi,\cdot)}_{L^\infty(\R)}    d\xi \\
        &\quad \leq 20 \exp(-\delta/\epsilon) \int_{\R^d} \norm{\xi}^2_2 \cG(\xi)    d\xi \\
        &\quad \leq 20 \exp(-\delta/\epsilon) {\cI_4^*}. 
    \end{split}
\end{equation}
Hence, we can combine all the previous to
\begin{equation}
    \norm{\partial_j\partial_\ell(u-u_\epsilon)}_{L^2_\mu(B_M^d)} \leq 81(6\pi_A(0)+5d)\cI_4^*\left[\frac{\delta^{3/2}}{\epsilon}\ln(\frac{1}{\delta})+\exp(-\delta/\epsilon)\right]
\end{equation}
We combine this inequality with the previously proven $W^{1,\infty}$-bound and use the inequality $\sqrt{1+d+d(1+d)/2}\leq \sqrt{3}d$ to tackle the sum inside the definition of the $H^2$-norm to find
\begin{equation}
    \begin{split}
\norm{u-u_\epsilon}_{H^2_\mu(B_M^d)} &\leq \sqrt{3}d\cdot 81(6\pi_A(0)+5d)M\cI_4^*\left[\epsilon + \frac{\delta^{3/2}}{\epsilon}\ln(\frac{1}{\delta})+\exp(-\delta/\epsilon)\right] \\
&\leq C_\gamma d(\pi_A(0)+d)M\cI_4^*\epsilon^\gamma,
    \end{split}
\end{equation}
where we $\delta = \epsilon^{\frac{2}{3}(1+\gamma)}$ with $\gamma\in (0,1/2)$ to recover the second inequality. 
\end{proof}

It is clear that if we can discretize the integrals in the definition of $u_\epsilon$ in an appropriate manner, then we have proven the existence of randomized neural networks that approximate $u$ well. We will use the following lemma that quantifies the accuracy of such a Monte Carlo-type approximation. 

\begin{lemma}\label{lem:MC}
Let $p\in[2,\infty)$, $q,m\in\mathbb{N}$, let $(\Omega, \mathcal{F}, \mathcal{P})$ and $(\mathcal{D}, \mathcal{A}, \mu)$ be probability spaces, and let for every $q\in\mathcal{D}$ the maps $X_i^q:\Omega\to\mathbb{R}, i\in\{1,\ldots, m\}$, be i.i.d. random variables with $\E{\abs{X^q_1}}<\infty$. Then it holds that
\begin{equation}\label{eq:MC-formula}
    \E{\left(\int_{\mathcal{D}}\abs{\E{X^q_1}-\frac{1}{m}\sum_{i=1}^m X^q_i}^p\mu(dq)\right)^{1/p}} \leq 2\sqrt{\frac{p-1}{m}} \left(\int_{\mathcal{D}}\E{\abs{\E{X^q_1}-X_1^q}^p}\mu(dq)\right)^{1/p}.
\end{equation}
\end{lemma}
\begin{proof}
The proof involves Hölder's inequality, Fubini's theorem and \cite[Corollary 2.5]{grohs2018proof}. The calculation is as in \cite[eq. (226)]{grohs2018proof}.
\end{proof}

Using this lemma, we can prove that \emph{generalized} randomized neural networks \eqref{def:rfm-1} can approximate functions of the form \eqref{eq:Hrepresentation2} well in Sobolev norm. In particular, the convergence rate in terms of the neural network width $N$ is dimension-independent and the curse of dimensionality is fully overcome as long as the constant $\cI_4$ \eqref{eq:IfiniteLinfty} does not depend exponentially on $d$. 

\begin{theorem}\label{thm:approx-general}
Let $u \colon \R^d \to \R$, let $M >0$ and assume there exists $G \colon \R^d \to \mC$ such that 
	\begin{equation}
	u(x) = \int_{\R^d} e^{i x \cdot \xi} G(\xi) d\xi
	\end{equation}
	for all $x \in [-M,M]^d$. Define
	\begin{equation}
	\label{eq:IfiniteLinfty}
	\cI_\ell := \max_{k\in \{0, \ldots, \ell \}}\left(\int_{\R^d} \|\xi\|^k [ 1+\bar{F}(M\|\xi\|_1) +  \bar{F}(1)-\bar{F}(-1) ] \frac{(\abs{G(\xi)}+\abs{G(-\xi)})^2}{ \pi_A(\xi)} d \xi\right)^{\frac{1}{2}}. 
	\end{equation}   
If Assumptions \ref{ass:dist-1}, \ref{ass:piA-radial} and \ref{ass:Fbar} are satisfied then for any $\gamma\in(0,1/2)$ and $N\in\N$ there exist a constant $C_\gamma$ and an $\R^N$-valued, $\sigma(A,B,Y)$-measurable random vector $W$ such that
 \begin{equation}
     \E{\norm{u-U^{A,B,Y}_W}_{H^1_\mu(B_M^d)}} \leq \frac{100Md\cI_{2}}{N^{1/4}} \quad \text{and}\quad \E{\norm{u-U^{A,B,Y}_W}_{H^2_\mu(B_M^d)}} \leq \frac{ C_\gamma d(\pi_A(0)+d)M\cI_4}{N^{\frac{\gamma}{2(2+\gamma)}}}.
 \end{equation}
\end{theorem}

\begin{proof}
First we note that the approximation $u_\epsilon$ from Proposition \ref{prop:general}, given by
\begin{equation}
    u_\epsilon(x) = \int_{\R^d} \int_{- \infty}^\infty \int_0^{B(\xi)} H_\epsilon(x \cdot \xi+u - y)  \alpha(\xi,u) d y d u d \xi.
\end{equation}
can be rewritten as
\begin{equation}
     u_\epsilon(x) = \mathbb{E}_{A,B,Y}\left[ H_\epsilon(x\cdot A + B - Y ) f(A,B)\right], \quad \text{where } f(\xi,u) = \frac{B(\xi)\alpha(\xi,u)}{\pi_A(\xi)\pi_B(u)}.
\end{equation}
If we define the random function $X$ as $X(x) := H_\epsilon(x\cdot A + B - Y ) f(A,B)$ then a straightforward approximation of $u(x)$ is given by $\sum_{i=1}^N X_i(x)$, where $X_1, \ldots, X_N$ are $N$ iid realizations of $X$. 

Next, letting $\bmalpha\in\N_0^d$ with $\ell := \norm{\bmalpha}_1 \leq 2$, setting $D:=B_M^d$ and invoking Lemma \ref{lem:MC} brings us that
\begin{equation}\label{eq:1-48}
\begin{split}
    \E{\norm{D^\bmalpha u_\epsilon - D^\bmalpha\frac{1}{N}\sum_{i=1}^N X_i}_{L^2_\mu(D)}} &\leq \frac{2}{\sqrt{N}} \left(\int_{D}\E{\abs{D^\bmalpha(u_\epsilon-X_1)}^2}\mu(dx)\right)^{1/2}\\ &\leq \frac{2}{\sqrt{N}} \left(\int_{D}\E{\abs{D^\bmalpha X_1}^2}\mu(dx)\right)^{1/2}. 
\end{split}
\end{equation}
We can estimate the integrand of the right-hand side as
\begin{equation}\label{eq:1-49}
    \begin{split}
        \E{\abs{D^\bmalpha X_1}^2} &= \E{\abs{D^\bmalpha H_\epsilon(x\cdot \xi + u - y ) f(\xi, u) }^2} \leq \frac{\|{\sigma^{(\ell)}}\|_\infty^2}{\epsilon^{2\ell}} \E{\abs{\norm{\xi}_2^\ell f(\xi, u) }^2}. 
    \end{split}
\end{equation}
We can further compute an upper bound as
\begin{equation}
\begin{split}
    \E{\abs{\norm{\xi}_2^\ell f(\xi, u) }^2} &=\int_{\R^d}\int_{\R} \norm{\xi}_2^{2\ell}\abs{f(\xi,u)}^2 \pi_A(\xi)\pi_B(u)dud\xi \\
&= \int_{\R^d}\int_{\R} \norm{\xi}_2^{2\ell} (\mathbbm{1}_{(-M \|\xi\|_1,0]}(u) + 4\mathbbm{1}_{[-1,1]}(u))^2 \frac{\cG(\xi)^2B(\xi)^2}{\pi_A(\xi)\pi_B(u)}dud\xi \\ 
&\leq 2\int_{\R^d}\int_{\R} \norm{\xi}_2^{2\ell} (\mathbbm{1}_{(-M \|\xi\|_1,0]}(u) + 16\mathbbm{1}_{[-1,1]}(u))\frac{\cG(\xi)^2(2M\norm{\xi}_1+1)^2}{\pi_A(\xi)\pi_B(u)}dud\xi \\
&\leq  2\cdot 16\cdot 9M^2d\max_{0\leq k \leq 2(\ell+1)} \int_{\R^d} \norm{\xi}_2^{k} (\bar{F}(dM\norm{\xi}_2) + \bar{F}(1)-\bar{F}(-1))\frac{\cG(\xi)^2}{\pi_A(\xi)} d\xi \\
&\leq 288M^2d \cI_{2(\ell+1)}^2, 
\end{split}
\end{equation}
where we used
\begin{equation}\label{eq:Bxi-bound}
    B(\xi)^2 \leq (2M\norm{\xi}_1+1)^2 \leq (2M\sqrt{d}\norm{\xi}_2+M\sqrt{d})^2.\leq 9M^2d\max\{\norm{\xi}_2,1\}^2
\end{equation}
The total error contributed to the Monte Carlo approximation can then be quantified as
\begin{equation}
    \E{\norm{D^\bmalpha u_\epsilon - D^\bmalpha\frac{1}{N}\sum_{i=1}^N X_i}_{L^2_\mu(D)}} \leq \frac{34M\sqrt{d}\cI_{2(\ell+1)}}{\epsilon^\ell\sqrt{N}}. 
\end{equation}
Combining this with Proposition \ref{prop:general} gives
\begin{equation}
    \E{\norm{u - \frac{1}{N}\sum_{i=1}^N X_i}_{H^1_\mu(D)}} \leq \frac{100\sqrt{d}}{68}\cdot 34M\sqrt{d}\cI_{2}\left[\epsilon + \frac{1}{\epsilon\sqrt{N}}\right] = \frac{100Md\cI_{2}}{N^{1/4}},
\end{equation}
where we set $\epsilon = N^{-1/4}$ and used that $\sqrt{1+d}\leq 100d/68$ to tackle the sum inside the definition of the $H^1$-norm. Similarly, we use that $\sqrt{1+d+d(1+d)/2}\leq \sqrt{3}d$ and Proposition \ref{prop:general} to find
\begin{equation}
\begin{split}
    \E{\norm{u - \frac{1}{N}\sum_{i=1}^N X_i}_{H^2_\mu(D)}} &\leq  C_\gamma d(\pi_A(0)+d)M\cI_4\left[\epsilon^\gamma + \frac{1}{\epsilon^2\sqrt{N}}\right] \leq \frac{ C_\gamma d(\pi_A(0)+d)M\cI_4}{N^{\frac{\gamma}{2(2+\gamma)}}},
\end{split}
\end{equation}
where we set $\epsilon = N^{-\frac{1}{2(2+\gamma)}}$. This concludes the proof of the result. 
\end{proof}

\subsection{Approximating Sobolev functions with randomized neural networks with uniformly distributed weights}\label{sec:approx2}

With Theorem \ref{thm:approx-general} we have a presented a very general approximation results for a slightly modified definition of randomized neural networks. In this section, we will adapt this result in the following manner: 
\begin{itemize}
    \item rather than using \emph{generalized} randomized neural networks \eqref{def:rfm-1}, we will use the default definition of randomized neural networks, namely
\begin{equation}\label{def:rfm-2}
    U^{A,B}_W(x) = W_0 + \sum_{i=1}^N W_i \sigma(A_i\cdot x+B_i);
\end{equation}
\item we will formulate approximation results for Sobolev functions, rather than functions that satisfy the somewhat intransparent condition \eqref{eq:Hrepresentation2};
\item we replace the general probability distribution that were subject to Assumptions \ref{ass:dist-1}, \ref{ass:piA-radial} and \ref{ass:Fbar} simply by uniform probability distributions. 
\end{itemize}
The final bullet point can be formalized as follows:
\begin{assumption}
The random variables $A_i$ and $B_i$ ($i=1\ldots,N$) are iid and satisfy the following requirements for some $M,R\geq 1$: 
\begin{itemize}
	\item the distribution $\Tilde{\pi_A}$ of $A_1$ is uniform on $B_R^d$ and
	\item the distribution $\Tilde{\pi_B}$ of $B_1$ is uniform on the interval $[-3M\sqrt{d}R-1,2]$. 
\end{itemize}
\end{assumption}
With the above goals in mind, we first prove an auxiliary result that gives an upper bound on the constants $\cI_k$ \eqref{eq:IfiniteLinfty} of Theorem \ref{thm:approx-general} in terms of the Sobolev norm of the considered function. 
\begin{lemma}\label{lem:sobolev}
Let $u$ be as in Theorem \ref{thm:approx-general} and let
$s>d/2$, $R\in[1,\infty)$ and $k\in\N$. If $u\in H^{k+s}(\R^d)$ then it holds that
    \begin{equation}
\left(\int_{\R^d\setminus B_R^d} \norm{\xi}^k_2  (\abs{G(\xi)}+\abs{G(-\xi)}) d\xi \right)^2\leq  \frac{4\omega_d \norm{u}_{H^{k+s}}^2}{(2\pi)^d(2s-d)R^{2s-d}}. 
    \end{equation}    
\end{lemma}
\begin{proof}
First, we notice that
\begin{equation}
    \int_{\R^d\setminus B_R^d} \norm{\xi}^k_2  (\abs{G(\xi)}+\abs{G(-\xi)}) d\xi = \frac{2}{(2\pi)^{\frac{d}{2}}} \int_{\R^d\setminus B_R^d} \norm{\xi}^k_2  \abs{\hat{u}(\xi)} d\xi. 
\end{equation}
In addition, we can calculate that
\begin{equation}
        \begin{split}
\left(\int_{\R^d\setminus B_R^d} \norm{\xi}^k_2  \abs{\hat{u}(\xi)} d\xi \right)^2&= \left(\int_{\R^d\setminus B_R^d} \norm{\xi}^k_2 \frac{(1+\norm{\xi}^2_2)^{s/2}}{(1+\norm{\xi}^2_2)^{s/2}} \abs{\hat{u}(\xi)} d\xi \right)^2 \\
&\leq \int_{\R^d\setminus B_R^d}(1+\norm{\xi}^2_2)^{-s} d\xi \cdot  \int_{\R^d\setminus B_R^d}(1+\norm{\xi}^2_2)^{s} \norm{\xi}^{2k}_2 \abs{\hat{u}(\xi)}^2 d\xi \\
&\leq \int_{\R^d\setminus B_R^d}\norm{\xi}^{-2s} d\xi \cdot  \int_{\R^d\setminus B_R^d}(1+\norm{\xi}^2_2)^{k+s}\abs{\hat{u}(\xi)}^2 d\xi \\
& = \omega_d \norm{u}_{H^{k+s}}^2 \int_R^\infty r^{-2s+d-1} dr = \omega_d \norm{u}_{H^{k+s}}^2 \frac{R^{-2s+d}}{2s-d}. 
        \end{split}
    \end{equation}
\end{proof}

We can now present the main result of the section: we use Theorem \ref{thm:approx-general} to prove that Sobolev functions of sufficient regularity can be approximated by default randomized neural networks with uniformly distributed weights can be well approximated in $H^1$-norm and $H^2$-norm. In addition, the curse of dimensionality can be fully overcome if sufficient smoothness is available. 

\begin{theorem}\label{thm:approx-sobolev}
Assume the conditions of Proposition \ref{prop:general} and assume that $u\in H^{4+s}(\R^d)$ for some $s\geq (d+1)/2$. Then for any $\gamma\in(0,1/2)$ there exist a constant $C_\gamma$ and an $\R^N$-valued, $\sigma(A,B)$-measurable random vector $W$ such that
 \begin{equation}
     \E{\norm{u-U^{A,B}_W}_{H^1_\mu(B_M^d)}} \leq \frac{C_\gamma M^2d^2\norm{u}_{H^{3+s}(\R^d)}}{N^{\frac{s-d/2}{6s+9}}} \quad \text{and}\quad \E{\norm{u-U^{A,B}_W}_{H^2_\mu(B_M^d)}} \leq \frac{ C_\gamma M^2d^2 \norm{u}_{H^{4+s}(\R^d)}}{N^{\frac{\gamma}{\gamma+2}\frac{s-d/2}{2s+3}}}.
 \end{equation}
\end{theorem}
\begin{proof}
We first recall that from the proof of Theorem \ref{thm:approx-general} it follows that for the function
\begin{equation}
    u_\epsilon(x) = \int_{\R^d} \int_{- \infty}^\infty \int_0^{B(\xi)} H_\epsilon(x \cdot \xi+u - y)  \alpha(\xi,u) d y d u d \xi.
\end{equation}
it holds that
\begin{equation}
    \norm{u-u_\epsilon}_{H^2_\mu(B_M^d)} \leq C_\gamma d(\pi_A(0)+d)M\cI_4^*\epsilon^\gamma
\end{equation}
We now adapt this approximation by restricting the outer integral to the ball $B_R^d = \{\xi\in \R^d: \norm{\xi}_2\leq R\}$, where $R\geq 1$, giving rise to the function
\begin{equation}\label{eq:u-epsilon-R}
    u_\epsilon^R(x) = \int_{B_R^d} \int_{- \infty}^\infty \int_0^{B(\xi)} H_\epsilon(x \cdot \xi+u - y)  \alpha(\xi,u) d y d u d \xi.
\end{equation}

\textit{Step 1: accuracy of $u_\epsilon^R$. }We now quantify the accuracy of this novel approximation. We let $D^\bmalpha$ be a differential operator with order $\ell:=\norm{\bmalpha}_1\leq 2$ and calculate using \eqref{eq:abs-alpha-bound} and \eqref{eq:Bxi-bound} that
\begin{equation}
\begin{split}
    \abs{D^\bmalpha( u_\epsilon(x)-u_\epsilon^R(x))} &\leq \int_{\R^d\setminus B_R^d}  \int_{- \infty}^\infty \int_0^{B(\xi)} \abs{\xi^\bmalpha H_\epsilon^{(\ell)}(x \cdot \xi+u - y)  \alpha(\xi,u)} d y d u d \xi \\
    &\leq \frac{\norm{\sigma^{(\ell)}}_\infty }{\epsilon^\ell} \int_{\R^d\setminus B_R^d} \norm{\xi}^{\ell} B(\xi) \int_{- \infty}^\infty \abs{\alpha(\xi,u)}  d u d \xi \\
    &\leq \frac{\norm{\sigma^{(\ell)}}_\infty 9M\sqrt{d}\cdot 3M\sqrt{d}}{\epsilon^\ell} \int_{\R^d\setminus B_R^d}  \norm{\xi}^{2+\ell} \cG(\xi) d \xi \\
    &= \frac{54M^2 d}{\epsilon^\ell } \int_{\R^d\setminus B_R^d}  \norm{\xi}^{2+\ell} \cG(\xi) d \xi. 
\end{split}
\end{equation}
We can combine this inequality with Lemma \ref{lem:sobolev} to find
\begin{equation}
    \abs{D^\bmalpha( u_\epsilon(x)-u_\epsilon^R(x))} \leq \frac{54M^2d\sqrt{\omega_d}\norm{u}_{H^{2+\ell+s}} }{\epsilon^\ell \sqrt{(2\pi)^d}\sqrt{2s-d}} R^{-s+d/2}.
\end{equation}
Assuming that $2s\geq d+1$ we can already find that
\begin{equation}\label{eq:1-65}
    \norm{u_\epsilon-u_\epsilon^R}_{H^\ell_\mu(B_M^d)} \leq \frac{C M^2d^2 \norm{u}_{H^{2+\ell+s}}}{\epsilon^\ell}R^{-s+d/2}.
\end{equation}

\textit{Step 2: Monte Carlo approximation. }We can also rewrite formula \eqref{eq:u-epsilon-R} as 
\begin{equation}
\begin{split}
    u_\epsilon^R(x) &= \int_{B_R^d} \int_{- \infty}^\infty \int^u_{u-B(\xi)} H_\epsilon(x \cdot \xi+v)  \alpha(\xi,u) d v d u d \xi \\
    &= \int_{B_R^d} \int_{- M\sqrt{d}R}^1 \int^2_{-3M\sqrt{d}R-1} H_\epsilon(x \cdot \xi+v)  \alpha(\xi,u) \mathbbm{1}_{[u-B(\xi), u]}(v) d v d u d \xi,
\end{split}
\end{equation}
such that it also holds that
\begin{equation}
    D^\bmalpha u_\epsilon^R(x) = \mathbb{E}_{\xi,v}\left[D^\bmalpha H_\epsilon(x\cdot \xi + v ) \Tilde{f}(\xi, v)\right],
\end{equation}
where
\begin{equation}
    \Tilde{f}(\xi,v) = 3\omega_d R^d (M\sqrt{d}R+1)\int_{- M\sqrt{d}R}^1  \alpha(\xi,u) \mathbbm{1}_{[u-B(\xi), u]}(v) du. 
\end{equation}
Using \eqref{eq:1-10} and \eqref{eq:1-17} we can upper bound this quantity as
\begin{equation}
    \abs{\Tilde{f}(\xi,v)} \leq 15\omega_d R^d (M\sqrt{d}R+1)^2 \cG(\xi)
\end{equation}
and calculate
\begin{equation}
    \begin{split}
        \E{\abs{\norm{\xi}_2^\ell \Tilde{f}(\xi, v) }^2} &=\int_{\R^d}\int_{\R} \norm{\xi}_2^{2\ell}\abs{\Tilde{f}(\xi,v)}^2 \Tilde{\pi_A}(\xi)\Tilde{\pi_B}(v)dvd\xi \\
        &\leq 125\omega_d R^d (M\sqrt{d}R+1)^3 \int_{B_R^d}\norm{\xi}_2^{2\ell} \cG(\xi)^2 d\xi \\
        &\leq 500\omega_d R^d (M\sqrt{d}R+1)^3 \norm{u}_{H^\ell(\R^d)}^2.
    \end{split}
\end{equation}
Hence, using a similar calculation as in \eqref{eq:1-48} and \eqref{eq:1-49} we find
\begin{equation}
\begin{split}
    \E{\norm{D^\bmalpha u_\epsilon^R - D^\bmalpha\frac{1}{N}\sum_{i=1}^N X_i}_{L^2_\mu(D)}} &\leq \frac{23\sqrt{\omega_d R^d}(M\sqrt{d}R+1)^{3/2}\norm{u}_{H^\ell(\R^d)}}{\epsilon^\ell\sqrt{N}} \\
    &\leq \frac{C d \sqrt{M^3 R^{d+3}}\norm{u}_{H^\ell(\R^d)}}{\epsilon^\ell\sqrt{N}}.
\end{split} 
\end{equation}

\textit{Step 3: total error. }If $2s\geq d+1$ we can combine the above bound with \eqref{eq:1-65} to find for $\ell\leq 1$ that
\begin{equation}
    \begin{split}
        \E{\norm{ u -\frac{1}{N}\sum_{i=1}^N X_i}_{H^1_\mu(D)}} &\leq C M^2d^2 \left(\frac{\norm{u}_{H^{2+\ell+s}(\R^d)}R^{d/2}}{\epsilon} \left[\frac{1}{R^{s}}+\frac{R^{3/2}}{\sqrt{N}}\right]+\cI_2^*\epsilon\right),
    \end{split}
\end{equation}
and for $\ell=2$ that
\begin{equation}
    \begin{split}
        \E{\norm{ u - \frac{1}{N}\sum_{i=1}^N X_i}_{H^2_\mu(D)}} &\leq C_\gamma M^2d^2(\pi_A(0)+d) \left(\frac{\norm{u}_{H^{4+s}(\R^d)}R^{d/2}}{\epsilon^2} \left[\frac{1}{R^{s}}+\frac{R^{3/2}}{\sqrt{N}}\right]+\cI_4^*\epsilon^\gamma\right). 
    \end{split}
\end{equation}

We see that these bounds depend on $\pi_A$ through $\pi_A(0)$ and $\cI^*_2$, resp. $\cI^*_4$. In this case we can choose $\pi_A$ freely as long as Assumption \ref{ass:dist-1} is satisfied. We will use
\begin{equation}
    \pi_A(\xi) = \frac{1}{2\omega_d}\begin{cases}
        1 & \norm{x}_2\leq 1, \\
        \norm{\xi}_2^{-d-1} & \norm{x}_2> 1. 
    \end{cases}
\end{equation}
With this definition we can also compute a concrete upper bound on the following constant:
\begin{equation}
\begin{split}
	\cI_\ell^* &= \max_{k\in \{0, \ldots, \ell \}}\left(\int_{\R^d} \|\xi\|^k \frac{(\abs{G(\xi)}+\abs{G(-\xi)})^2}{ \pi_A(\xi)} d \xi\right)^{\frac{1}{2}}\\ &\leq 2 \max_{k\in \{0, \ldots, \ell \}}\left(\int_{\R^d} \|\xi\|^{k+d+1}(\abs{\hat{u}(\xi)} d \xi\right)^{\frac{1}{2}} \leq 2 \norm{u}_{H^{(\ell+d+1)/2}} \leq 2 \norm{u}_{H^{2+\ell+s}},
 \end{split}
 \end{equation}
where we used $2s\geq d+1$ to find that $(\ell+d+1)/2\leq 2+\ell+s$. 

Finally, in both cases it follows that the optimal choice is $R = N^{\frac{1}{2s+3}}$ and $\epsilon = R^{-\frac{s-d/2}{\gamma+2}}$, where one has to set $\gamma=0$ for $\ell\leq 1$. We find
\begin{equation}
    \begin{split}
        \E{\norm{u - \frac{1}{N}\sum_{i=1}^N X_i}_{H^1_\mu(D)}} &\leq CM^2d^2\norm{u}_{H^{3+s}(\R^d)} N^{-\frac{s-d/2}{4s+6}}\\
        \E{\norm{u - \frac{1}{N}\sum_{i=1}^N X_i}_{H^2_\mu(D)}} &\leq C_\gamma M^2d^2 \norm{u}_{H^{4+s}(\R^d)} N^{-\frac{\gamma}{\gamma+2}\frac{s-d/2}{2s+3}}.
    \end{split}
\end{equation}

\end{proof}

\begin{remark}
The approximation rates in Theorem \ref{thm:approx-sobolev} depend on $s$, $d$ and $\gamma$, which makes them rather hard to interpret. Below we show that in the optimal case of high smoothness and the optimal choice of $\gamma$ the convergence rate in $H^1$-norm tends towards $1/4$ and that in $H^2$ towards $1/10$, 
\begin{equation}
    H^1: \quad \frac{s-d/2}{4s+6} \xrightarrow[s\to\infty]{} \frac{1}{4}\qquad \text{and}\qquad H^2: \quad \frac{\gamma}{\gamma+2}\frac{s-d/2}{2s+3} \xrightarrow[s\to\infty]{}  \frac{1}{2} \cdot \frac{\gamma}{2+\gamma}\xrightarrow[\gamma\to 1/2]{} \frac{1}{10}. 
\end{equation}
This also shows that the adaptation of Theorem \ref{thm:approx-general} to Sobolev functions and uniform distributed weights comes at a hefty cost in terms of regularity. In Theorem \ref{thm:approx-sobolev} the Sobolev regularity $s$ has to be considerably larger than $d$ to obtain a dimension-independent rate, whereas in Theorem \ref{thm:approx-general} the convergence rate did not depend on $s$. On the positive side, we observe that in the limit $s\to\infty$ we exactly retrieve the convergence rates from Theorem \ref{thm:approx-general}. 
\end{remark}

\section{Numerical results}\label{sec:4}
With a solid justification of the approximation capacities of randomized neural networks in place, we now demonstrate empirically that using randomized neural networks \ref{sec:random-nn} within physics-informed ELMs (Section \ref{sec:pielm}) can indeed achieve extremely high accuracy in very short times, even in high-dimensional settings. For this reason, we focus on PDEs which can typically be high-dimensional such as the heat equation, the Black-Scholes equations and the Heston model. 
\subsection{Heat equation}\label{sec:heat}

Let $\Omega \subset \mathbb{R}^d$ be a bounded domain, we consider the following problem
\begin{align}
		u_t(x,t)- \Delta u(x,t) &= f(x,t) \;\, \forall x \in \Omega,\, t\in(0,T), \label{prb1} \\
		u(x,t) &= g(x,t) \;\, \forall x \in \partial \Omega, \, t\in(0,T),\label{bc1} \\
		u(x,0) &= h(x)\quad \,\forall x \in \Omega. \label{bc2}
\end{align}

We consider randomized neural networks as introduced in Section \ref{sec:random-nn} with the activation function $\sigma$ being the hyperbolic tangent (tanh) or sigmoid function. Let
\begin{equation}
    u_W(x,t) = \sum_{i=1}^N W_i \sigma(A_i\cdot (x,t)+B_i)= W \cdot \sigma(A\cdot(x,t)+b), 
\end{equation}
where $A \in \mathbb{R}^{N\times (d+1)}$ and $b\in\mathbb{R}^N$ are randomly generated, $W \in \mathbb{R}^N$ needs to be solved, and $(x,t)$ is chosen from the space-time domain $\Omega \times (0,T)$.

Following the PIELM framework of Section \ref{sec:pielm}, we first randomly generate collocation points $\{(x_p,t_p)\}_{1}^{N_{int}} \in \Omega\times (0,T)$,  $\{(x_k,t_k)\}_{1}^{N_{sb}} \in  \partial \Omega \times (0,T)$ and  $\{x_l\}_{1}^{N_{tb}} \in \Omega $ according to the uniform distribution. Then one can calculate that the linear system \eqref{eq:pielm} in this case is given by,
\begin{align}
    W \cdot \left(\partial_t \sigma(A\cdot(x_p,t_p)+b) - \Delta \sigma(A\cdot(x_p,t_p)+b)\right) &= f(x_p,t_p) \;\;\;\; \text{for} \  p = 1,\cdots,N_{int}, \label{rf1} \\
	W \cdot \sigma(A\cdot(x_k,t_k)+b) &= g(x_k,t_k)  \quad \text{for} \  k = 1,\cdots,N_{sb}, \label{rf2} \\
	W \cdot \sigma(A\cdot(x_l,0)+b)  &= h(x_l)  \quad \;\;\;\;\;\text{for} \  l = 1,\cdots,N_{tb}. \label{rf3}, 
\end{align}
which we solve in a least-square sense. In our experiment we then consider the following problem for different values of $d$. 

\begin{example}\label{exam1} Given $\Omega = [0,1]^d$, T = 1. We consider the problem \eqref{prb1}-\eqref{bc2} with exact solution $u = \frac{\Vert x\Vert^2}{d} + 2t$.
\end{example}

We randomly generate $A$ and $b$ from the uniform distribution $\mathcal{U}(-0.01,0.01)$. To study the impact of different randomized neural networks, $N$ is set to be 800, 1600 and 3200. Although automatic differentiation could be used, we find that it is computationally more efficient and also more accurate to use the difference method. We set the spacing to $10^{-6}$ for first derivatives and $10^{-3}$ for second derivatives. As for the least-square solver, we use {\verb+scipy.linalg.lstsq+}  in Python,
relying on QR decomposition to directly solve the least-squares problem. Finally, we follow the examples in \cite{siddhartha2022pinns}, and let $N_{int} = 8192$, $N_{sb} = 2048$, $N_{tb} =6144 $. Table \ref{tab1} and Table \ref{tab11} shows the resulting $L^2$ errors (computed on a test set of $10^5$ randomly chosen samples) and running time for different dimensions $d$. We observe that when using $\tanh$ as the activation function, the $L^2$-errors  range from $10^{-4} \%$ and $2.4\%$, with computation times varying between  1 and 29 seconds for various dimension $d$. In contrast, using a sigmoid activation function improves the accuracy, achieving an error of $0.9\%$ when $d=100$.

\begin{table}[!ht]
\renewcommand{\arraystretch}{1.2}
\setlength\tabcolsep{4.2mm}
	\begin{center}		
		\begin{tabular}{|c|c|c|c|c|c|c|}
			\hline
			$d$ & $N_{int}$ & $N_{sb}$ &  $N_{tb}$   & $N$ & $\Vert u-u^{A,b}_W\Vert_{L^2}$ & Time (s) \\
			\hline
			\multirow{3}{*}{5} &	\multirow{3}{*}{8192}& \multirow{3}{*}{2048}&\multirow{3}{*}{6144} & 800&0.00027\%&1\\		
			\cline{5-7}  
			&&&&1600&0.00015\%&3\\
			\cline{5-7}  
			&&&&3200&0.000085\%&10\\
			\hline		
			\multirow{3}{*}{10} &	\multirow{3}{*}{8192}& \multirow{3}{*}{2048}&\multirow{3}{*}{6144} &
			800& 0.0010\%&2\\
			\cline{5-7}  
			&&&&1600&0.00044\%&4\\
			\cline{5-7}   
			&&&&3200&0.00027\%&11\\
			\hline
			\multirow{3}{*}{20} &	\multirow{3}{*}{8192}& \multirow{3}{*}{2048}&\multirow{3}{*}{6144} &
			800& 1.3\%&2\\
			\cline{5-7}    
			&&&&1600&0.57\%&5\\
			\cline{5-7}   
			&&&&3200&0.0054\%&13\\
			\hline		
			\multirow{3}{*}{50} &	\multirow{3}{*}{8192}& \multirow{3}{*}{2048}&\multirow{3}{*}{6144}&
			800&3.3\%&4\\
			\cline{5-7}    
			&&&&1600&2.5\%&8\\
			\cline{5-7}    
			&&&&3200&1.5\%&18\\
			\hline		
			\multirow{3}{*}{100} &	\multirow{3}{*}{8192}& \multirow{3}{*}{2048}&\multirow{3}{*}{6144}&
			800&3.9\%&6\\
			\cline{5-7}   
			&&&&1600&2.8\%&13\\
			\cline{5-7}  
			&&&&3200&2.4\%&29\\
			\hline		
		\end{tabular}
	\end{center}
	\caption{$\sigma =  \textbf{tanh}$, $L^2$ errors of using different randomized neural networks for d-dimensional heat equation.}
	\label{tab1}
\end{table}

\begin{table}[!ht]
	\renewcommand{\arraystretch}{1.2}
	\setlength\tabcolsep{4.2mm}
	\begin{center}		
		\begin{tabular}{|c|c|c|c|c|c|c|}
			\hline
			$d$ & $N_{int}$ & $N_{sb}$ &  $N_{tb}$   & $N$ & $\Vert u-u^{A,b}_W\Vert_{L^2}$ & Time (s) \\
			\hline
			\multirow{3}{*}{5} &	\multirow{3}{*}{8192}& \multirow{3}{*}{2048}&\multirow{3}{*}{6144} & 800&0.0011\%&1\\		
			\cline{5-7}  
			&&&&1600&0.00079\%&3\\
			\cline{5-7}  
			&&&&3200&0.00074\%&10\\
			\hline		
			\multirow{3}{*}{10} &	\multirow{3}{*}{8192}& \multirow{3}{*}{2048}&\multirow{3}{*}{6144} &
			800& 0.0090\%&2\\
			\cline{5-7}  
			&&&&1600&0.0018\%&4\\
			\cline{5-7}   
			&&&&3200&0.00091\%&11\\
			\hline
			\multirow{3}{*}{20} &	\multirow{3}{*}{8192}& \multirow{3}{*}{2048}&\multirow{3}{*}{6144} &
			800& 0.54\%&2\\
			\cline{5-7}    
			&&&&1600&0.21\%&5\\
			\cline{5-7}   
			&&&&3200&0.0095\%&13\\
			\hline		
			\multirow{3}{*}{50} &	\multirow{3}{*}{8192}& \multirow{3}{*}{2048}&\multirow{3}{*}{6144}&
			800&1.4\%&4\\
			\cline{5-7}    
			&&&&1600&1.0\%&8\\
			\cline{5-7}    
			&&&&3200&0.52\%&18\\
			\hline		
			\multirow{3}{*}{100} &	\multirow{3}{*}{8192}& \multirow{3}{*}{2048}&\multirow{3}{*}{6144}&
			800&1.8\%&6\\
			\cline{5-7}   
			&&&&1600&1.3\%&13\\
			\cline{5-7}  
			&&&&3200&0.94\%&29\\
			\hline		
		\end{tabular}
	\end{center}
	\caption{$\sigma = \textbf{sigmoid}$, $L^2$ errors of using different randomized neural networks for d-dimensional heat equation.}
	\label{tab11}
\end{table}

\subsection{Black-Scholes equations}\label{sec:bs}

Next, we use randomized neural networks to solve the Black-Scholes model with uncorrelated noise. For this reason we let $W = (W^1, \ldots, W^d): [0,T]\times \Omega\to\R^d$ be a $d$-dimensional standard Brownian motion, we let $\mu\in \R$ be the stock return and $\epsilon_1, \ldots, \epsilon_d$ (positive) stock volatilities. By using Itô's formula and the Feynman-Kac theorem, it holds that for a payoff function $\psi$ that the following function,
\begin{equation}\label{eq:sol-bs}
u(x,t) = \mathbb{E}[\psi([x^i \mathrm{exp}((\mu-\frac{1}{2}\epsilon_i^2 )t + \epsilon_iW_t^i)]_{i=1}^d)],
\end{equation}
is the solution of the PDE problem
\begin{align}
	\partial_t u(x,t) &= \frac{1}{2} \sum_{i=1}^{d}  |\sigma_i x^i|^2 (\partial^2_{x_ix_i}u)(x,t) + \sum_{i=1}^{d} \mu x^i (\partial_{x_i}u)(x,t) ,\;\,  \forall x \in \Omega,\, t\in(0,T), \label{prb2} \\
	u(x,t) &= \frac{1}{N_s}\sum_{n = 1}^{N_s} \psi([x^iexp((\mu-\frac{1}{2}\epsilon_i^2 )t + \epsilon_iW_t^{i,n})]_{i=1}^d),\;\;\;\; \;\;\;\forall x \in \partial \Omega, \, t\in(0,T),\label{bc3} \\
	u(x,0) &= \psi(x) \qquad \qquad\qquad \qquad\qquad \qquad\qquad \qquad \quad \quad \;\;\;\,\forall x \in \Omega. \label{bc4}
\end{align}
In the above, we imposed \emph{approximate} Dirichlet boundary conditions \eqref{bc3} in the sense that we replaced exact Dirichlet boundary conditions (in terms of an expectation value as in \eqref{eq:sol-bs} by a sample mean of size $N_s$, i.e. by averaging over $N_s$ realizations of the Brownian motion \cite{RT}. 


Then we can obtain the optimal weights of a randomized  neural network by solving a least-square problem with the following linear system.
\begin{align}
	&W \cdot \Big(\partial_t \sigma(A\cdot(x_p,t_p)+b) - \frac{1}{2} \sum_{i=1}^{d}  |\sigma_i x_p^i|^2 (\partial^2_{x_ix_i} \sigma(A\cdot(x_p,t_p)+b))  \notag \\
	&- \sum_{i=1}^{d} \mu x_p^i (\partial_{x_i}\sigma(A\cdot(x_p,t_p)+b)) \Big) = 0 \quad \text{for} \  p = 1,\cdots,N_{int}, \label{rf4} \\
	&\beta_1 W \cdot \sigma(A\cdot(x_k,t_k)+b) = \frac{\beta_1}{N_s}\sum_{n = 1}^{N_s} \psi\Big([x^i_k\mathrm{exp}\big((\mu-\frac{1}{2}\epsilon_i^2 )t_k + \epsilon_iW_k^{i,n}\big)]_{i=1}^d\Big) \quad \text{for} \  k = 1,\cdots,N_{sb}, \label{rf5} \\
	&\beta_2W \cdot \sigma(A\cdot(x_l,0)+b)  = \beta_2\psi(x_l)  \quad \text{for} \  l = 1,\cdots,N_{tb}, \label{rf6}
\end{align}
where $\beta_1$ and $\beta_2$ are hyperparameters that balance the scaling in the above linear system. In our experiment we then consider the following problem for different values of $d$. 

\begin{example}\label{exam2} Given $\Omega = [90,110]^d$, $T = 1$. Let $\mu = -0.05$ and $\sigma_i = \frac{1}{10}+\frac{i}{200}$, $i\in \{1,\cdots,d\}$. We consider the problem \eqref{prb2}-\eqref{bc4} with the initial condition $\psi(x) = max(max_{1\leq i\leq d}(x_i) - 100,0)$.
\end{example}
In this experiment, we add a normalization to the input data by $\hat{x} = (x-90)/20 $ and $\hat{t} =t$. Then  $A$ and $b$ are randomly generated with a uniform distribution $\mathcal{U}(-0.1,0.1)$, and $\sigma$ is the sigmoid function or tanh function. The  size of the sample mean in the boundary condition is set to be $N_s = 16384$. Let $N_{int} = 32768$, $N_{sb} =16384$ and $N_{tb} =16384$.  Table \ref{tab3} and Table \ref{tab4} present the relative $L^2$-errors (computed by $\mathcal{E}_{G}^r = \frac{\Vert u-u^{A,b}_W\Vert_{L^2}}{\Vert u\Vert_{L^2}}$ on a test set of $10^5$ randomly chosen samples) for different dimensions $d$. The relative $L^2$-error when $d = 100$ is approximately three percent, with a computation time of six minutes, demonstrating the efficiency of using randomized neural networks.
\begin{table}[!ht]
	\begin{center}
		\renewcommand{\arraystretch}{1.2}
		\setlength\tabcolsep{2.5mm}
		\begin{tabular}{|c|c|c|c|c|c|c|c|c|c|}
			\hline
			$d$ & $N_{int}$ & $N_{sb}$ &  $N_{tb}$ &$N_s$ & $N$& $\beta_1$&$\beta_2$& $\mathcal{E}_{G}^r $ & Time (s) \\
			\hline
			1&	32768&16384&16384 &16384& 800 & 5 & 10 &0.9\%&11\\		
			\hline		
			2&	32768&16384&16384 &16384& 800 & 5 & 10 &0.8\%&12\\		
			\hline		
			10&	32768&16384&16384 &16384& 800 & 5 & 10 &1.1\%&15\\		
			\hline		
			20&	32768& 16384&16384 &16384& 3200 & 5 & 10 &1.0\%&120\\		
			\hline		
			50&	32768&16384&16384 &16384& 3200 & 5 & 100 &1.7\%&162\\		
			\hline		
			100&32768& 16384&16384 &16384& 3200 & 5 & 100 &3.1\%&246\\		
			\hline		
			
		\end{tabular}
	\end{center}
	\caption{$\sigma =  \textbf{tanh}$, relative $L^2$ errors of using randomized neural networks for d-dimensional Black-Scholes model.}
	\label{tab3}
\end{table}

\begin{table}[!ht]
	\begin{center}
		\renewcommand{\arraystretch}{1.2}
		\setlength\tabcolsep{2.5mm}
		\begin{tabular}{|c|c|c|c|c|c|c|c|c|c|}
			\hline
			$d$ & $N_{int}$ & $N_{sb}$ &  $N_{tb}$ &$N_s$ & $N$& $\beta_1$&$\beta_2$& $\mathcal{E}_{G}^r $ & Time (s) \\
			\hline
			1&	32768&16384&16384 &16384& 800 & 5 & 10 &1.1\%&11\\		
			\hline		
			2&	32768&16384&16384 &16384& 800 & 5 & 10 &1.0\%&13\\		
			\hline		
			10&	32768&16384&16384 &16384& 800 & 5 & 10 &1.1\%&15\\
			\hline		
			20&	32768& 16384&16384 &16384& 3200 & 5 & 10 &1.0\%& 130\\		
			\hline		
			50&	32768&16384&16384 &16384& 3200 & 5 & 100 &2.0\%& 165\\		
			\hline		
			100&32768& 16384&16384 &16384& 3200 & 5 & 100 &3.2\%& 244\\		
			\hline		
			
		\end{tabular}
	\end{center}
	\caption{$\sigma = \textbf{sigmoid}$, relative $L^2$ errors of using randomized neural networks for d-dimensional Black-Scholes model.}
	\label{tab4}  
\end{table}

\subsection{Heston model}\label{sec:heston}

Finally, we demonstrate that randomized neural networks can also be used to solve the Heston model. Following the notation of
\cite{RT}, we set $\delta\in \N$, $d=2\delta$, $\alpha,\beta\in\R$ and $\theta, \kappa>0$ to introduce the multi-dimensional Heston model satisyfing the Feller condition $2\kappa\theta>\beta^2$. 

By using It\"o formula and Feynman-Kac theorem, let $\psi$ be a payoff function and $\delta = d/2$, then
\begin{align}
	u(x,t) &=\mathbb{E}\bigg[\psi    \bigg(  \sum_{i=1}^{\delta}\Big(\big[x^{2i-1}\mathrm{exp}\big((\alpha - \frac{x^{2i}}{2})t + W_t^{2i-1}\sqrt{x^{2i}}\big)\big]e_{2i-1} \notag \\ 
	&+ \Big[max\{\big[max\{\frac{\beta}{2}\sqrt{t}, max\{\frac{\beta}{2}\sqrt{t},\sqrt{x^{2i}}\}  + \frac{\beta}{2} (\rho W^{2i-1}_t + \sqrt{1-\rho^2}W^{2i}_t)\}\big]^2 
	\\ &+ (\kappa\theta - \frac{\beta^2}{4} - \kappa x^{2i})t, 0 \}\Big]e_{2i} \Big) \bigg)\bigg]. \notag
\end{align}

is the solution of the following PDE problem
\begin{align}
	\partial_t u(x,t) &- \sum_{i=1}^{\delta} \Big[\alpha x^{2i-1}(\partial_{x_{2i-1}}u)(x,t) + \kappa(\theta -x^{2i})(\partial_{x_{2i}}u)(x,t) \Big]  \notag
	\\ &- \sum_{i=1}^{\delta}  \Big[\frac{|x^{2i}|}{2}(x^{2i-1})^2(\partial^2_{x_{2i-1}x_{2i-1}}u)(x,t) \label{prb3}\\
	&+ 2x^{2i-1}\beta\rho(\partial^2_{x_{2i-1}x_{2i}}u)(x,t) + \beta^2( \partial^2_{x_{2i}x_{2i}}u)(x,t)\Big]\;\;\forall x \in\Omega, \, t\in(0,T), \notag   \\
	u(x,t)&=\frac{1}{N_s}\sum_{n = 1}^{N_s}\psi   \bigg(  \sum_{i=1}^{\delta}\Big(\big[x^{2i-1}\mathrm{exp}\big((\alpha - \frac{x^{2i}}{2})t + W_t^{2i-1}\sqrt{x^{2i}}\big)\big]e_{2i-1} \notag \\ 
	&+ \Big[max\{\big[max\{\frac{\beta}{2}\sqrt{t}, max\{\frac{\beta}{2}\sqrt{t},\sqrt{x^{2i}}\}  + \frac{\beta}{2} (\rho W^{2i-1}_t + \sqrt{1-\rho^2}W^{2i}_t)\}\big]^2  \label{bc5}
	\\ &+ (\kappa\theta - \frac{\beta^2}{4} - \kappa x^{2i})t, 0 \}\Big]e_{2i} \Big) \bigg)
	\qquad\qquad\qquad\qquad\;\;\forall x \in \partial \Omega, \, t\in(0,T),  \notag \\
	u(x,0) &= \psi(x) \qquad \qquad\qquad \qquad\qquad \qquad\qquad \qquad \;\;\;\qquad \forall x \in \Omega, \label{bc6}
\end{align}
where Dirichlet boundary condition \eqref{bc5} is imposed by the sample mean.

Similarly, we can obtain the optimal weights of a randomized  neural network by solving a least-square problem with the corresponding linear system.

\begin{example}\label{exam3} Given $\Omega = _{i=1}^{d/2}([90,110] \times [0.02,0.2])$, $T = 1$.  Let $\alpha = 1/20$,  $\beta = 1/5$, $\kappa = 6/10$, $\theta = 1/25$, $\rho = -1/5 $.  We consider the problem \eqref{prb3}-\eqref{bc6} with the initial condition $\psi(x) = max(110-\frac{2}{d}\sum_{i=1}^{d/2}x_{2i-1},0).$
\end{example}
In this experiment, we apply the same normalization strategy, and use the 
tanh function as the activation function. The values of $A$ and $b$ are randomly generated with a uniform distribution $\mathcal{U}(-0.1,0.1)$,  The sample size is kept the same as in Example  \ref{sec:bs}. Finally, Table \ref{tab5} presents the relative $L^2$-errors for different dimensions $d$, highlighting the efficiency of using randomized neural networks.

\begin{table}[!ht]
	\begin{center}
		\renewcommand{\arraystretch}{1.2}
		\setlength\tabcolsep{2.5mm}
		\begin{tabular}{|c|c|c|c|c|c|c|c|c|c|}
			\hline
			$d$ & $N_{int}$ & $N_{sb}$ &  $N_{tb}$ &$N_s$ & $N$& $\beta_1$&$\beta_2$& $\mathcal{E}_{G}^r $ & Time (s) \\
			\hline		
			2&	32768&16384&16384 &16384& 800 & 800 & 800 &2.4\%&12\\	
			\hline		
			4&	32768&16384&16384 &16384& 800 & 5 & 50 &1.4\%&14\\	
			\hline		
			10&	32768&16384&16384 &16384& 800 & 5 & 10 &1.1\%&16\\		
			\hline		
			30&	32768& 16384&16384 &16384& 3200 & 5 & 10 &1.3\%&160\\		
			\hline		
			50&	32768&16384&16384 &16384& 3200 & 10 & 100 &1.8\%&220\\		
			\hline		
			100&32768& 16384&16384 &16384& 3200 & 10 & 100 &2.9\%&370\\		
			\hline		
			
		\end{tabular}
	\end{center}
	\caption{$\sigma = \tanh$, relative $L^2$ errors of using randomized neural networks for d-dimensional Heston model.}
	\label{tab5}
\end{table}

\section{Conclusion}

In this work, we present theoretical results and numerical experiments demonstrating the potential of using randomized neural networks in efficiently solving high-dimensional PDEs. Our main contribution is the proof that randomized neural networks can approximate functions in higher-order Sobolev norms at dimension-independent convergence rates, overcoming the curse of dimensionality. Specifically, we establish upper bounds on the approximation error for functions in $H^1$ and $H^2$ norms, highlighting the effectiveness of these networks in solving linear first-order and second-order PDEs. Our experimental results confirm that randomized neural networks provide high accuracy and low computational cost, even for complex problems such as the heat equation, Black-Scholes model, and Heston model, with errors and computation times well within acceptable limits. This work expands the existing literature by showing that randomized neural networks can achieve high efficiency and accuracy in practical high-dimensional PDE applications, paving the way for their broader use in scientific computing.

\bibliographystyle{abbrv}
\bibliography{ref}

\end{document}